\documentclass[a4paper,10pt]{article}
\usepackage[utf8]{inputenc}
\usepackage{amsmath}
\usepackage{amsthm}
\usepackage{amsfonts}

\usepackage{amssymb}
\usepackage{mathrsfs}
\usepackage{geometry}
\usepackage{verbatim}
\usepackage{tikz}
\usetikzlibrary{automata,positioning}
\title{Subword Complexity and (non)-automaticity of certain completely multiplicative functions}
\author{YINING HU \\
CNRS, Institut de Math\'ematiques de Jussieu-PRG \\
Universit\'e Pierre et Marie Curie, Case 247 \\
4 Place Jussieu \\
F-75252 Paris Cedex 05 (France) \\
{\tt yining.hu@imj-prg.fr}}
\date{}

\begin{document}
\maketitle
\newtheorem{lem}{Lemma}
\newtheorem{thm}{Theorem}
\newtheorem{coro}{Corollary}
\theoremstyle{definition}
\newtheorem{define}{Definition}

\begin{abstract}
In this article, we prove that for a completely multiplicative function $f$ from $\mathbb{N}^*$ to a field $K$ such that the set
$$\{p \;|\; f(p)\neq 1_K \;\mbox{and }p \mbox{ is prime}\}$$
is finite, the asymptotic subword complexity of $f$ is $\Theta(n^t)$, where $t$ is the number of primes 
$p$ that $f(p)\neq 0_K, 1_K$. This proves in particular that
sequences like $((-1)^{v_2(n)+v_3(n)})_n$ are not $k$-automatic for $k\geq 2$.

\end{abstract}

\section{Introduction}
The subject of this article is the subword complexity and (non-)automaticity of certain completely multiplicative functions. 
From the definition we see that a completely multiplicative function is completely
determined by its value on prime numbers. 

We recall the definition and  a few results about subword complexity and automaticity.
The proof of Theorem \ref{up} can be found for example 
in Chapter 10 of \cite{as},
and the proof of Theorem \ref{Cobham} can be found in \cite{complexity}.

\begin{define}
 Let $\mathbf{u}$ be a infinite sequence of symbols from an alphabet. We define the subword complexity $p_{\mathbf{u}}(n)$ of $u$ to be the number of different factors 
 (consecutive letters) of length $n$ in $\mathbf{u}$.
\end{define}

\begin{thm}[Morse and Hedlund \cite{mh}]\label{up}
 A sequence $\mathbf{u}$ is ultimately periodic if and only if there exists a non-negative integer $k$ such that $p_\mathbf{u}(k)\leq k$.
\end{thm}

We recall one of the equivalent definitions of an $k$-automatic sequence.

\begin{define}
 Let $k\geq 2$ be an integer. A sequence $\mathbf{u}$ is called $k$-automatic if there is a finite set of sequences containing $\mathbf{u}$ and closed under 
 the operations
 $$(v(n))_n \rightarrow (v(kn+r))_n, \;\; \mbox{  for }r = 0,...,k-1. $$
\end{define}

\begin{lem}\label{coding}
 Let $\mathbf{u} = (u_n )_{n\geq0}$ be a $k$-automatic sequence, and let $\rho$ be a coding.
Then the sequence $\rho(u)$ is also $k$-automatic.
\end{lem}

We recall some asymptotic notations for asymptotic complexity: \\
$f(n)=O(g(n))$:  $f$ is bounded above by $g$ (up to constant factor) asymptotically;\\
$f(n)=\Omega(g(n))$:  $f$ is bounded below by $g$ (up to constant factor) asymptotically;\\
$f(n)=\Theta(g(n))$: $f$  is bounded both above and below (up to constant factors) by $g$ asymptotically.

\begin{thm}\label{Cobham}\emph{(Cobham)}
If a sequence $\mathbf{u}$ is $k$-automatic, then $p_\mathbf{u}(n)=O(n)$.
 
\end{thm}

\begin{thm}[Cobham]\label{kl}
Let $k,l\geq 2$ be two multiplicatively independent integers, and suppose the sequence $\mathbf{u}=(u(n))_n\geq 0$ is both $k$- and
$l$-automatic, then $\mathbf{u}$ is ultimately periodic.
\end{thm}

As a first example we can consider a completely multiplicative sequence $(u(n))_n$ that takes values
in $\mathbb{R}$ such that $u(2)=-1$, and $u(p)=1$ for 
all $p\in \mathbb{P}\backslash \{2\}$, where $\mathbb{P}$ denotes the set of prime numbers. Then we have for all $n\in\mathbb{N}^*$, $u(n)=(-1)^{v_2(n)}$, 
where $v_p$ is the $p$-adic valuation. This sequence is $2$-automatic, for we have
$$u(2n)=-u(n),$$
$$u(2n+1)=1.$$
It is easy to see that $\mathbf{u}$ is not ultimately periodic. Thus by Theorem \ref{up} and Theorem \ref{Cobham}, 
we have $p_{\mathbf{u}}(n)=\Theta(n)$.
In general, for all prime $p$, the sequence $((-1)^{v_p(n)})_{n\geq 1}$ is $p$-automatic and has asymptotic subword complexity $\Theta(n)$.

For a more interesting example, consider the completely multiplicative sequence 
$\mathbf{a}=(a(n))_{n\geq 1}$ taking values in $\mathbb{R}$ such that $a(2)=a(3)=-1$, and $a(p)=1$ for 
$p\in \mathbb{P}\backslash\{2,3\}$. Intuitively this sequence cannot be $k$-automatic: if $\mathbf{a}$ is $2$-automatic then there is no reason why it should
not be $3$-automatic, but $\mathbf{a}$ is not ultimately periodic (which we will prove later), so by Theorem \ref{kl}, $\mathbf{a}$ cannot be at the same time
$2$- and $3$-automatic. It can be shown that $\mathbf{a}$ is not $6$-automatic, and there is no reason for $\mathbf{a}$ to be $k$-automatic for 
$k$ other than $2$, $3$ or $6$ either. Indeed, we will prove in the Section 2 that the subword complexity of $\mathbf{a}$ is $\omega(n^2)$, which implies by 
Theorem \ref{Cobham} that it cannot be $k$-automatic. In Section 3 we a general result.

\section{The example of $(-1)^{v_2(n)+v_3(n)}$}
In this section we prove that the asymptotic subword complexity of the sequence $(u(n)))_{n\geq 1}=((-1)^{v_2(n)+v_3(n)})_{n\geq 1}$ is $\Theta(n^2)$. 

It is easy to see that since $\mathbf{u}$ is the product of two sequences $((-1)^{v_2(n)})_{n\geq 1}$ and 
$((-1)^{v_3(n)})_{n\geq 1}$, both of which have asymptotic subword complexity $\Theta(n)$, the asymptotic subword complexity of $\mathbf{u}$ is $O(n^2)$.
Indeed, we have the following lemma:

\begin{lem}\label{O}

Let the function $f_0(n)$   be the product of functions $f_1(n)$, $f_2(n)$,..., $f_k(n)$. 
Then for all $n\in\mathbb{N}$, $p_{f_0}(n)\leq \prod\limits_{i=1}^k p_{f_i}(n)$. 
\end{lem}
\begin{proof}
Let $F_{i,n}$ be the set of factors of length $n$ of $f_i$. There is an surjection from $\prod\limits_{i=1}^k F_{i,n}$ to $F_{0,n}$.
\end{proof}

We recall B\'ezout's Lemma in the form that we need:
\begin{lem}[B\'ezout]\label{bz}
 Let $n$ be an integer and let $p$ and $q$ be coprime integers. Then
 $$\{k\cdot p-l\cdot q\;|\; k,l\in\mathbb{Z} \mbox{ and } k,l>n \}=\mathbb{Z}.$$
\end{lem}
The usual form of B\'ezout's Lemma says that if $p$ and $q$ are coprime integers, then
 $$\{k\cdot p-l\cdot q\;|\; k,l\in\mathbb{Z} \}=\mathbb{Z}.$$
We may assume that $k,l>n$ by replacing, when this is not the case,  $(k,l)$ by $(k+m\cdot q,l+m\cdot p)$, for an integer $m$ large
enough.

Now we prove that $p_{\mathbf{u}}(n)=\Omega(n)$.

First, in order to isolate $v_2(n)$ and $v_3(n)$, we consider the subsequences $(a(n))_n=(u(3n+1))_n=((-1)^{v_2(3n+1)})_n$ and 
$(b(n))_n=(u(2n+1))_n=((-1)^{v_3(2n+1)})_n$. By Lemma \ref{periodic} in Section 3 we know that $\mathbf{a}$ and $\mathbf{b}$ are not
ultimately periodic. Therefore, $p_{\mathbf{a}}(n)> n$ and $p_{\mathbf{b}}(n)> n$ by Theorem \ref{up}. This mean that there exist $n+1$
factors of length $n$ in $\mathbf{a}$ (resp. $\mathbf{b}$), which we denote by $A_0, A_1, ..., A_n$ (resp. $B_0, B_1, ..., B_n$) and
by $\alpha_0,\alpha_1,...,\alpha_n$ (resp. $\beta_0, \beta_1,...,\beta_n$) their starting position in $\mathbf{a}$ (resp. $\mathbf{b}$).
We choose an
integer $N$ such that these factors are contained in the initial segment of length $N$ in $\mathbf{a}$ or $\mathbf{b}$ respectively.


In the next step, for $(i,j)\in \{0,1,...,n\}^2$, we want the scattered subwords $A_i$ and $B_j$ to
occur ``in the same place'' in the original sequence $\mathbf{u}$, such that they would form distinct factors $U_{ij}$ of $\mathbf{u}$. While
this does not happen for all couples $(i,j)$, we will show in the following lines that
it is true for at least $\lceil (n+1)^2/6 \rceil$ couples in $\{0,1,...,n\}^2$. First we remark that the factors $A_0, A_1, ..., A_n$
(resp. $B_0, B_1, ..., B_n$)
recur periodically in $\mathbf{a}$ (resp. $\mathbf{b}$). This is due to the property of the $p$-adic valuation that
$$v_p(m)>v_p(n)\Rightarrow v_p(m+n)=v_p(n).$$
Thus, let $M$ be an integer such that $M>v_2(3n+1)$ and $M>v_3(2n+1)$ for all $n\in \{0,1,...,N-1\}$. Then for all integer $k$ and all 
$n\in \{0,1,...,N-1\}$, we have $\mathbf{a}(n+k\cdot 2^M)=\mathbf{a}(n)$ and $\mathbf{b}(n+k\cdot 3^M)=\mathbf{b}(n)$. Therefore the
set of starting positions of the factor $A_i$ in $\mathbf{a}$ include
$$\{\alpha_i+k\cdot 2^M| k\in \mathbb{N}\},$$
and the set of starting positions of the factor $B_j$ in $\mathbf{b}$ include
$$\{\beta_j+l\cdot 3^M| l\in \mathbb{N}\}.$$
In the original sequence $\mathbf{u}$, the set of starting positions of the scattered subword $A_i$ include
$$\{3\alpha_i+3k\cdot 2^M+1| k\in \mathbb{N}\},$$
and the set of starting positions of the scattered subword $B_j$ in $\mathbf{u}$ include
$$\{2\beta_j+2l\cdot 3^M+1| l\in \mathbb{N}\}.$$
By Lemma \ref{bz}, the set 
\begin{equation}\label{eq:6}
\{3k\cdot2^M-2l\cdot 3^M|k,l\in \mathbb{N}\}=\{6(k\cdot2^{M-1}-l\cdot3^{M-1}) |k,l\in \mathbb{N}\}=\{6z | z \in\mathbb{Z}\}
\end{equation}

There exist $r\in \{0,1,...,5\}$ such that 
$$\mbox{card}(\{(i,j)\in \{0,...,n\}^2 \;|\; 2\beta_j-3\alpha \equiv r \mod 6 \})\geq \lceil(n+1)^2/6\rceil.$$
This, combined with equation \ref{eq:6}, means that for at least $\lceil(n+1)^2/6\rceil$ couples $(i,j)$ , there is an occurence of $A_i$
and an occurence of $B_j$ in $\mathbf{u}$, such $A_i$ starts $r$ letters before $B_j$. 
This gives at least $\lceil(n+1)^2/6\rceil$ distinct factors of $\mathbf{u}$ of length $\max\{3n-2, r+2n-1\}$, which proves that the
asymptotic subword complexity of $\mathbf{u}$
is $\Omega(n^2)$, which implies by Theorem \ref{Cobham} that $\mathbf{u}$ is not automatic.

\section{The general case}

To study the general case, we need the following generalization of Lemma \ref{bz}:

\begin{lem}\label{bzz}
Let $q_0$, $q_1$,...,$q_n$ be pairwise coprime integers. Let $l_1$,...,$l_n$  be integers. Then exists positive integers
$k_0$, $k_1$,...,$k_n$ such that for all $i\in \{1,...,n\}$, $k_0q_0-k_iq_i=l_i$.
\end{lem}
\begin{proof}
We prove this by induction. 

By Lemma \ref{bz} we know that there exist $k_0$ and $k_1\in \mathbb{N}$ such that
$k_0q_0-k_1q_1=l_1$.

Suppose that for an integer $j\in\{1,...,n\}$, we have proven that there exist $k_0$, $k_1$,...,$k_j$ such that
for all $i\in \{1,...,j\}$,  $k_0q_0-k_iq_i=l_i$. If $j=n$ we are done. Otherwise by the pairwise coprime assumption
and Lemma \ref{bz} we know that
there exists $k\in\mathbb{N}$ and $k_{j+1}'\in \mathbb{N}$ such that
$$k(q_0q_1...q_j)-k_{j+1}'q_{j+1}=l_{j+1}-k_0q_0.$$
That is, 
$$q_0(k_0+k(q_1...q_j))-k_{j+1}'q_{j+1}=l_{j+1}$$
For $i=0,1,...,j$, we define $k_i'$ as $k_i+k\cdot\prod\limits_{m=0,m\neq i}^{j}q_m$. Then $k_0'$, $k_1'$,...,$k_{j+1}'$ satisfy the condition
that for all $i\in\{1,...,j+1\}$, $k_0'q_0-k_i'q_i=l_i$.

\end{proof}

We introduce the following lemma before proving our main theorem.
\begin{lem}\label{periodic}
 Let $a$ be an element of finite order different from the identity element in a multiplicative group $G$. 
 Let $p$ be a prime number. Let $q$ and $b$ be positive integers such that $p\nmid q$.
Then the sequence $(u(n))_n:=(a^{v_p (qn+b)})_{n\geq 0}$ is not ultimately periodic. 
\end{lem}
\begin{proof}
Suppose that $(u(n))_n$ were ultimately periodic and let $T$ be a period of $(u(n))_n$. 
We write $T$ as $p^k\cdot T'$ where $k=v_p(T)$ and $p\nmid T'$.
We claim that there exists an integer $m$ larger than the length of the initial non-periodic segment of $\mathbf{u}$ such that 
 $v_p(qm+b)=k+1$. In fact, since $q$ and $p^{k+1}$ are coprime, by Lemma \ref{bz} we know that there exists integers $m,n$ large enough
 to be in the periodic part of $\mathbf{u}$ such that
 $$n\cdot p^{k+1}-m\cdot q=b.$$
Furthermore, we can assume that $p\nmid n$, by eventually replacing $(m,n)$ by 
$(m+p^{k+1},n+q)$ when this is not the case. Thus we have $v_p(qm+b)=v_p(n\cdot p^{k+1})=k+1$. We also have $v_p(q(m+T)+b)=v_p(qm+b+qT)=k$. Therefore
$u(m)=a^{k+1}\neq a^{k}=u(m+T)$ since $a$ is not the identity element, which contradicts the definition of $m$ and $T$. Therefore $(u(n))_n$ cannot be ultimately
periodic.
\end{proof}

\begin{thm}\label{main}
Let $\mathbf{u}=(u(n))_{n\geq 1}$ be a completely multiplicative sequence taking values in a field $K$. We suppose that the number of prime
numbers $p$ such that $u(p)\neq 1_K$ is finite. Let $t$ be the number of primes $p$ such that $u(p)\neq 1_K, 0_K$. Then the asymptotic subword complexity 
of $\mathbf{u}$ is $\Theta(n^t)$.
\end{thm}
\begin{proof}
If $t=0$, then $\mathbf{u}$ is periodic. Then $p_{\mathbf{u}}(n)$ is ultimately constant.

If not, we denote by $P=\{p_1,...,p_t\}$ the set of primes $p$ such that $u(p)\neq 1_K, 0_K$ and
by $q$ the product of primes where $\mathbf{u}$ takes the value $0_K$. 
The sequence $\mathbf{u}$ is the product of an ultimately periodic sequence and non-ultimately periodic automatic sequences : 
$$u(n)= z(n)\prod\limits_{p\in P} (u(p))^{v_p(n)}, $$
where $(z(n))_n$ is the sequence defined as $z(n)=0$ if and only if $u(n)=0$, $z(n)=1_K$ otherwise. By Lemma \ref{O} and Theorem \ref{Cobham}
we know that $p_\mathbf{u}=O(n^t)$.

We consider the subsequence 
$(w(n))_n=(u(qn+1))_n$. To prove that $p_\mathbf{u}(n)=\Omega(n^t)$ we only have to prove that $p_{\mathbf{w}}(u)=\Omega(n^t)$.

If $t=1$, we denote by $p$ the prime such that $u(p)\neq 1_K, 0_K$ and by $a$ the value of $u(p)$. Then $w(n)=a^{v_p(qn+1)}$. By Lemma 
\ref{periodic} we know that $\mathbf{w}$ is not ultimately periodic. So by Theorem \ref{up} we have
$p_\mathbf{w}(n)\geq n+1$. Therefore $p_\mathbf{u}=\Omega(n)$.

If $t\geq 2$, for $i=1,...,t$ we define $q_i$ to be  $\prod\limits_{j=1, j\neq i}^{t} p_j$, and we consider the sequences 
$(w_i(n))_n:=(w(q_in))_n=(u(qq_in+1))_n=(u(p_i)^{v_{p_i}(qq_i+1)})_n$. By Lemma \ref{periodic} and Theorem \ref{up}, 
$\mathbf{w_i}$ is not ultimately periodic so $\mathbf{w_i}$ at least $n+1$ factors of length $n$, which we denote by $W_{i,j}$ for $j=0,...,n$,
and by $\alpha_{i,j}$ their staring position in $\mathbf{w_i}$. We choose an integer $N$ such that for all $i=1,...,t$, and all $j=0,...,n$,
$W_{i,j}$ occurs in the initial segment of $\mathbf{w_i}$ of length $N$. There exists an integer $M$ such that for all $k=0,1,...,N-1$, all 
$i=0,...,t$, and all $m\in \mathbb{N}$, $w_i(k+m\cdot p_i^M)=w_i(k)$. Thus the starting position of $W_{i,j}$ in $\mathbf{w}$ include the set
$$\{q_i\alpha_{i,j}+q_imp_i^M\;|\; m\in \mathbb{N}\}.$$
There exist $r_2,...,r_t\in \{0,1,...,p_1...p_t\}$ such that the cardinality of the set
$$\{(j_1,...,j_t)\in \{0,1,...,n\}^t\;|\;(q_1\alpha_{1,j_1}-q_2\alpha_{2,j_2},...,q_1\alpha_{1,j_1}-q_t\alpha_{t,j_t})\equiv(r_2,...,r_t)
\mod p_1...p_t\}$$
is at least $\lceil n^t/(p_1^{t-1}p_2...p_t)\rceil$. By Lemma \ref{bzz} we know that 
$$\{(q_1m_1p_1^M-q_2m_2p_2^M,...,q_1m_1p_1^M-q_t m_t p_t^M)\;|\;m_1,...,m_t\in\mathbb{N}\}=(p_1...p_t\mathbb{Z})^{t-1}$$
Therefore $\mathbf{w}$ has at least $\lceil n^t/(p_1^{t-1}p_2...p_t)\rceil$ factors of length $(n\cdot \max\{q_1,...,q_t\}+p_1...p_t)$. This means 
that $p_\mathbf{w}(n)=\Omega(n)$ and therefore $p_\mathbf{u}(n)=\Omega(n)$.
\end{proof}

The following corollary is an immediate consequence of Theorem \ref{Cobham} and \ref{main}.
\begin{coro}\label{na}
 Let $(u(n))n$ be a completely multiplicative sequence taking values in a field $K$. We suppose that the number of prime numbers $p$ such that
$u(p)\neq 1_K$ is finite and that among them, for at least two primes $p_1$ and $p_2$ we have $(p_1)\neq 0_K$ and $u(p_2) \neq 0_K$. 
Then $(u(n))_n$ is not $k$-automatic for any $k\geq 2$.
\end{coro}

\begin{coro}
  Let $(u(n))n$ be a completely multiplicative sequence taking values in a field $K$. We suppose that the number of prime numbers $p$ such that
$u(p)=0_K$ is finite and that there exists an integer $d$ such that the cardinality of the set $\{p\in\mathbb{P}\;|\;u(p)^d\neq 0_k,1_K \}$ is finite and 
at least 2. Then $\mathbf{u}$ is not $k$-automatic for any $k\geq 2$.
\end{coro}
\begin{proof}
 By the assumption, the sequence $(u(n)^d)_n$ satisfies the conditions of Corollary \ref{na} and therefore is not $k$-automatic. This implies that by
 $\mathbf{u}$ is not $k$-automatic by Lemma \ref{coding}.
\end{proof}

 In the case where $K$ is the field of complex numbers, there is a more elegant proof of Corollary \ref{na} using the following result in \cite{am}  
 about automatic sequences and Dirichlet series:
 
 Let $k\geq 2$ be an integer and let $(u(n))_n$ be a $k$-automatic sequence with values in $\mathcal{C}$, then the Dirichlet series
 $$\sum\limits_{n=1}^{\infty} \frac{u(n)}{n^s} $$ 
 has a meromorphic continuation to the whole complex plane, whose poles (if any) are located at the points
 $$s=\frac{\log \lambda}{\log k}+\frac{2im\pi}{\log k}-l+1,$$
 where $\lambda$ is any eigenvalue of a certain matrix defined from the sequence $u$, where $m\in \mathbb{Z}$, $l\in \mathbb{N}$, and log 
 is a branch of the complex logarithm.

 \begin{proof}[Proof of Corollary \ref{na} for $K=\mathbb{C}$]
  We denote by $P$ the set of prime numbers. By the assumption of the corollary we know that
  $$P=A\cup B \cup \{p_1, p_2\},$$
  where $A=\{p\in P| u(p)=1\}$, $B=P\backslash(A\cup \{p_1, p_2\}) $.
  
  For $s$ such that $\Re(s)>1$, we consider the Dirichlet series
  \begin{align*}
   \sum\limits_{n=1}^{\infty} \frac{u(n)}{n^s} &= \prod\limits_{p\in P} \sum\limits_{i=0}^{\infty} \left(\frac{u(p)}{p^s}\right)^i \\
   &= \prod\limits_{p\in P}  \frac{1}{1-\frac{u(p)}{p^s}}  \\
   &= \left (\prod\limits_{p\in A}  \frac{1}{1-\frac{1}{p^s}} \right ) \left( \prod\limits_{q\in B}  \frac{1}{1-\frac{u(q)}{q^s}} \right)
   \left( \frac{1}{1-\frac{a}{p_1 ^s}} \right ) \left( \frac{1}{1-\frac{b}{p_2 ^s}} \right ) \\
   &=\left(\prod\limits_{p\in P}   \frac{1}{1-\frac{1}{p^s}} \right ) \left(  \prod\limits_{q\in B} \frac{q^s-1}{q^s-u(q)} \right )
   \frac{p_1 ^s -1}{p_1 ^s -a} \frac{p_2 ^s -1}{ p_2 ^s -b}\\
   &=\zeta(s) \left(  \prod\limits_{q\in B} \frac{q^s-1}{q^s-u(q)} \right )  
   \frac{p_1 ^s -1}{p_1 ^s -a} \frac{p_2 ^s -1}{ p_2 ^s -b}.\\
  \end{align*}
 Since the Riemann Zeta function $\zeta(s)$ is meromorphic on the whole complexe plane and the product over $B$ in the last line is finite,
 the Dirichet series $\sum\limits_{n=1}^{\infty} \frac{u(n)}{n^s}$ has meromorphic continuation on the whole complexe plane. Now we examine the
 poles of this function. The assumption of automaticity and multiplicativity implies that $a$ and $b$ are roots of unity, therefore 
 the poles of $\frac{p_1 ^s -1}{p_1 ^s -a} \frac{p_2 ^s -1}{ p_2 ^s -b}$
  are $\{\frac{(\arg(a)+2n\pi)i}{\log(p_1)}| n\in \mathbb{Z}\}\cup
  \{\frac{(\arg(b)+2m\pi)i}{\log(p_2)}|m\in \mathbb{Z}\}$. Since the $\zeta(s)$ has no zeros on the imaginary axis, these are also poles of the series
  $\sum\limits_{n=1}^{\infty} \frac{u(n)}{n^s}$. Therefore $(u(n))_n$ cannot be $k-$automatic according to the result cited above.

  \end{proof}

  \theoremstyle{remark}

\end{document}